\documentclass[a4paper,12pt]{amsart}
\usepackage{amsmath,amssymb}

\title[Real orbits on complex symmetric spaces]
{Orbits of real semisimple Lie groups \\ on real loci of complex symmetric spaces}

\author[S.~Cupit-Foutou]{St\'ephanie Cupit-Foutou}
\address{Ruhr-Universit\"at Bochum, Fakult\"at f\"ur Mathematik, D-44780 Bochum, Germany}
\email{stephanie.cupit@rub.de}

\author[D.~A.~Timashev]{Dmitry A. Timashev}
\thanks{The second author was partially supported by the Russian Foundation for Basic Research grant 16-01-00818}
\address{Lomonosov Moscow State University, Faculty of Mechanics and
Mathematics, Department of Higher Algebra, 119991 Moscow, Russia}
\email{timashev@mccme.ru}


\date{September 4, 2017}

\newcommand{\0}[2]{{}_{#2}#1}
\newcommand{\1}[2]{{}^{#2}#1}
\newcommand{\ii}{\boldsymbol{i}}
\newcommand{\CC}{\mathbb{C}}
\newcommand{\RR}{\mathbb{R}}
\newcommand{\ZZ}{\mathbb{Z}}
\newcommand{\g}{\mathfrak{g}}
\newcommand{\h}{\mathfrak{h}}
\newcommand{\s}{\mathfrak{s}}
\newcommand{\sgl}{\mathfrak{sl}}
\newcommand{\so}{\mathfrak{so}}
\newcommand{\sym}{\mathfrak{sym}}
\newcommand{\ttt}{\mathfrak{t}}
\newcommand{\z}{\mathfrak{z}}
\renewcommand{\phi}{\varphi}
\newcommand{\Ad}{\operatorname{Ad}}
\newcommand{\ad}{\operatorname{ad}}
\newcommand{\diag}{\operatorname{diag}}
\newcommand{\Gal}{\operatorname{Gal}}
\newcommand{\Img}{\operatorname{Im}}
\newcommand{\Ker}{\operatorname{Ker}}
\newcommand{\crit}[1]{(#1)^{\text{crit}}}
\newcommand{\Ho}[1]{\mathrm{H}^{#1}}
\newcommand{\Z}[1]{\mathrm{Z}^{#1}}
\newtheorem{theorem}{Theorem}[section]
\newtheorem{lemma}[theorem]{Lemma}
\newtheorem{corollary}[theorem]{Corollary}
\theoremstyle{definition}
\newtheorem{example}[theorem]{Example}
\newtheorem*{remark}{Remark}
\numberwithin{equation}{section}

\begin{document}

\begin{abstract}
Let $G$ be a complex semisimple algebraic group and $X$ be a complex symmetric homogeneous $G$-variety. Assume that both $G$, $X$ as well as the $G$-action on $X$ are defined over real numbers. Then $G(\RR)$ acts on $X(\RR)$ with finitely many orbits. We describe these orbits in combinatorial terms using Galois cohomology, thus providing a patch to a result of A.~Borel and L.~Ji.
\end{abstract}

\maketitle

\section*{Introduction}

In many classification problems of algebra and geometry, it often happens that a classification of objects of a given type over a non-algebraically closed field $\Bbbk$ is finer than over the algebraic closure $\overline\Bbbk$: an equivalence class over $\overline\Bbbk$ may split into several (even possibly infinitely many) equivalence classes over $\Bbbk$. One classical example when such a situation occurs is given by the classification of quadratic forms over real and complex numbers.

In the language of algebraic transformation groups, this situation is usually formalized as follows. Let $G$ be an algebraic group and $X$ be a homogeneous $G$-variety, both defined over $\Bbbk$. Then $G(\overline\Bbbk)$ acts on $X(\overline\Bbbk)$ transitively, while $G(\Bbbk)$ may have several orbits on $X(\Bbbk)$. A natural problem is to classify these orbits. One standard tool for solving this kind of problem is Galois cohomology. In what follows we restrict ourselves to the case $\Bbbk=\RR$, $\overline\Bbbk=\CC$, as in the above example. In this case, the general Galois cohomology theory is greatly simplified.

In particular, suppose that $G$ is a connected semisimple algebraic group defined over $\RR$ and $X$ is a symmetric homogeneous $G$-variety. This means that the stabilizer of any point of $X$ is a symmetric subgroup of $G$, i.e., coincides, up to connected components, with the fixed point subgroup of an involution on $G$. The classification of orbits on the real locus of a symmetric homogeneous variety was studied by A.~Borel and L.~Ji in \cite[Chap.\,6]{comp-symm}. Unfortunately, some essential arguments are missing therein and  Theorem II.6.5 in loc.~cit.\ does not hold in general.

In this note, we give a corrected description of $G(\RR)$-orbits in $X(\RR)$. Our approach follows that of \cite[Chap.\,6]{comp-symm}; it is based on Galois cohomology. The most important difference is that our description of orbits involves an action of the normalizer of a certain subtorus by twisted conjugation instead of the usual action of the respective Weyl group in \cite[Thm.\,II.6.5]{comp-symm}. Thus this note may be considered as a patch for \cite[Chap.\,6]{comp-symm}.

The description of orbits on the real locus of a symmetric homogeneous variety which we obtain can be reformulated in more geometric terms without any use of Galois cohomology. From this geometric perspective, it is natural to address the problem of extending the classification of real orbits from symmetric varieties to a more general class of spherical homogeneous spaces. We address this question in our work in progress; this will be developed elsewhere.

The paper is organized as follows. In Section~\ref{notation} we introduce basic notation and conventions which we use throughout the paper. In Section~\ref{symm} we collect some basic facts about symmetric spaces defined over real numbers (which are sometimes called semisimple symmetric spaces). They include conjugacy of $\RR$-split tori, the description of the little Weyl group, and some product decompositions of groups related to symmetric spaces. Generalities on Galois cohomology are recalled in Section~\ref{Galois}. Our main result on the description of orbits in real loci of symmetric spaces is obtained in Section~\ref{results}. Theorem~\ref{Ker} gives such a description in terms of Galois cohomology; our geometric interpretation of this statement can be read in Theorem~\ref{slice-orb}.

\subsection*{Acknowledgement}

This work was initiated during D.~Timashev's visit to the Ruhr University of Bochum in July 2016. The second author thanks the Transformation Groups Research team of Bochum for its hospitality and providing excellent working conditions. Both authors are grateful to P.~Heinzner for his support. Many thanks are due to the referee for helpful comments and suggestions aimed at improving the presentation.

\section{Notation, terminology, and conventions}
\label{notation}

\subsection{\relax}

The ground field is $\CC$. Complex algebraic varieties and groups are identified with the sets of their complex points. For an algebraic variety $Y$ defined over $\RR$ its set of real points is denoted by $Y(\RR)$. The real structure on $Y$ is determined by the action of the Galois group $\Gal(\CC/\RR)$ on $Y(\CC)=Y$, i.e., by the complex conjugation map, which is denoted as $y\mapsto\bar{y}$ unless otherwise specified. Note that $y\in Y(\RR)$ if and only if $y=\bar{y}$.

Algebraic groups are denoted by capital Latin letters and their Lie algebras are denoted by respective lowercase German letters. The identity component of an algebraic group $L$ is denoted by $L^{\circ}$.

\subsection{\relax}

Given a semisimple connected linear algebraic group $G$ defined over $\RR$, we consider a \emph{symmetric} homogeneous $G$-variety $X$ defined over $\RR$. ``Symmetric'' means that the stabilizer $H=G_{x_0}$ of any point $x_0\in X$ is a symmetric subgroup of $G$, i.e., $H^{\circ}=(G^{\theta})^{\circ}$, where $G^{\theta}\subset G$ is the subgroup of fixed points of an involutive automorphism $\theta$ of $G$.

If $G$ is simply connected (to which case we can always reduce by taking a universal cover of $G$), then $G^{\theta}$ is connected \cite[8.1]{end-lag} and $H$ is any subgroup between $G^{\theta}$ and the normalizer in $G$ of $G^\theta$. At the other extreme, if $G$ is adjoint, then $G^{\theta}$ is self-normalizing \cite[2.2]{emb(symm)} and $H$ is any subgroup of $G^{\theta}$ containing the identity component $(G^{\theta})^{\circ}$.

\subsection{\relax}

We assume the real locus $X(\RR)$ is not empty and choose a base point $x_0\in X(\RR)$. Then $H=G_{x_0}$ is defined over $\RR$. This assumption fits into the setting of \cite[Chap.\,6]{comp-symm}, since Borel and Ji start with a so-called semisimple real symmetric space and take its complexification.

Let $\sigma$ denote the real structure (complex conjugation) on $G$, i.e., $\sigma(g)=\bar{g}$, $\forall g\in G$. Then $G^{\sigma}=G(\RR)$ is the Lie subgroup of real points in $G$. It is connected whenever $G$ is simply connected \cite[Chap.\,4, Thm.\,2.2]{LieG}. The real structure on $X\simeq G/H$ is given by $$x=gx_0\mapsto\bar{x}=\sigma(g)x_0.$$

\subsection{\relax}

Choose a Cartan involution $\tau$ on $G$ compatible with $H$ and with the real structure. This is always possible, see, e.g., \cite[Chap.\,4, 3.3, Cor.\,3]{LieG}. The subgroups $G^{\tau}$ and $H^{\tau}$ are compact real forms of $G$ and $H$, respectively. The involutions $\theta,\sigma,\tau$ pairwise commute.

Let $\omega=\sigma\tau$, $\theta^*=\theta\omega$, $\sigma^*=\sigma\theta$, $\tau^*=\tau\theta$ be the other involutions generated by $\theta,\sigma,\tau$. The involutions $\theta,\omega,\theta^*$ are holomorphic whereas $\sigma,\tau,\sigma^*,\tau^*$ are antiholomorphic. Put $H^*=G^{\theta^*}$, yet another symmetric subgroup of $G$ defined over $\RR$.

Let $\g^{\pm\sigma,\pm\theta,\pm\tau}$ denote the common eigenspaces for $\sigma,\theta,\tau$ acting on $\g$ with eigenvalues $\pm1$ for the respective choices of signs. They are pairwise orthogonal with respect to the Killing form of $\g$ considered as a real Lie algebra. In the picture below, these eight eigenspaces are represented by boxes as follows: an indication $\iota$ or $-\iota$ in a box means that an involution $\iota$ acts on the respective eigenspace with eigenvalue $1$ or $-1$, respectively.
\begin{equation}\label{8boxes}
\begin{array}{|c|c|c|c|c|}
\cline{4-5}
        \multicolumn{3}{c|}{}                       &      -\sigma,\tau        &       -\sigma,\tau       \\
        \multicolumn{3}{r|}{\ii\cdot\g(\RR)=}       & \theta,-\omega,-\theta^* & -\theta,-\omega,\theta^* \\
\cline{4-5}
        \multicolumn{3}{c|}{}                       &      -\sigma,-\tau       &      -\sigma,-\tau       \\
\cline{1-2}
      \sigma,-\tau       &       \sigma,-\tau      &&  \theta,\omega,\theta^*  & -\theta,\omega,-\theta^* \\
\cline{4-5}
\theta,-\omega,-\theta^* & -\theta,-\omega,\theta^* &             \multicolumn{3}{c}{}                    \\
\cline{1-2}
      \sigma,\tau        &       \sigma,\tau        &             \multicolumn{3}{l}{=\g(\RR)}            \\
 \theta,\omega,\theta^*  & -\theta,\omega,-\theta^* &             \multicolumn{3}{c}{}                    \\
\cline{1-2}
\end{array}
\end{equation}
We specify eigenvalues only for those involutions which are essential for the sequel. Similar notation is used for eigenspaces of other commuting sets of involutions.

\subsection{\relax}

Let $T\simeq(\CC^{\times})^r$ be a complex algebraic torus. There is a unique \emph{polar decomposition} $T=T^+\times T^-$, where $T^+\simeq(\RR^+)^r$, $T^-\simeq(S^1)^r$, $T^-$ is the maximal compact subgroup of $T$, and $\ttt^-=\ii\ttt^+$.

For any subgroup $S\subseteq T$, we denote by $\0S2\subseteq S$ the set of elements of order $\le2$ and by $S^2$ the set of squares in $S$; both sets are subgroups and $\0S2$ is the kernel of the epimorphism $S\to S^2$, $s\mapsto s^2$. For instance, $\0T2=\0{T^-}2\simeq\{\pm1\}^r$.

Given an involutive (anti)holomorphic automorphism $\phi$ of $T$, we say that $T$ is \emph{$\phi$-split} if $\phi$ acts on $T^-$ as the inversion. In the holomorphic case, this means that $\phi$ acts on the whole $T$ as the inversion. In the antiholomorphic case, $\phi$ fixes $T^+$ pointwise and defines the real structure on $T$ such that $T(\RR)=T^+\times\0{T^-}2\simeq(\RR^{\times})^r$. In this case, we also say that $T$ is \emph{$\RR$-split}, assuming that the real structure is fixed.

A reductive algebraic group over $\RR$ is \emph{$\RR$-split} if it contains an $\RR$-split maximal torus (defined over $\RR$).

\section{Generalities on symmetric spaces}
\label{symm}

In this section, we collect some general results on the structure of symmetric spaces defined over real numbers. Our basic reference is \cite{ss-symm}.

\subsection{Conjugacy of $\RR$-split tori}

Choose a maximal Abelian real Lie subalgebra $\ttt_1^+\subseteq\g^{\sigma,-\theta,-\tau}$ and extend it to a maximal Abelian real Lie subalgebra $\ttt^+=\ttt_0^+\oplus\ttt_1^+\subseteq\g^{\sigma,-\tau}$. (Here $\ttt_p^+$ is the $(-1)^p$-eigenspace of $\theta$ in $\ttt^+$.) Let $\ttt=\ttt_0\oplus\ttt_1$ be the complexification of $\ttt^+$. The positions of these subalgebras with respect to the decomposition \eqref{8boxes} are indicated in the following picture:
$$
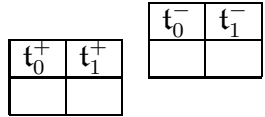
\begin{array}{|c|c|c|c|c|}
\cline{4-5}
        \multicolumn{3}{c|}{}                       &      \ttt_0^-       &      \ttt_1^-       \\
\cline{1-2}\cline{4-5}
      \ttt_0^+       &       \ttt_1^+      &&    &  \\
\cline{1-2}\cline{4-5}
         &             &             \multicolumn{3}{l}{}            \\
\cline{1-2}
\end{array}
$$

\begin{theorem}\label{conj}
The Lie algebra $\ttt$ (resp.\ $\ttt_1$) is an $\RR$-split maximal toral subalgebra in $\g^{-\omega}$ (resp.\ in $\g^{-\theta,-\omega}$). All maximal toral subalgebras in $\g^{-\omega}$ (resp.\ in $\g^{-\theta,-\omega}$) are conjugate to $\ttt$ (resp.\ $\ttt_1$) under $G^{\omega}$ (resp.\ $H^{\omega}$). All maximal $\RR$-split toral subalgebras in $\g$ (resp.\ in $\g^{-\theta}$) are conjugate to $\ttt$ (resp.\ $\ttt_1$) under $G^{\sigma}$ (resp.\ $H^{\sigma}$). Finally, all maximal $\RR$-split toral subalgebras in $\g^{-\omega}$ (resp.\ in $\g^{-\theta,-\omega}$) are conjugate to $\ttt$ (resp.\ $\ttt_1$) under $G^{\sigma,\tau}$ (resp.\ $H^{\sigma,\tau}$).
\end{theorem}

\begin{proof}
The first claim about $\ttt$ and $\ttt_1$ is clear, since $\ttt^+$ (hence $\ttt$) consists of semisimple elements, and both $\ttt$ and $\ttt_1$ are self-centralizing in $\g^{-\omega}$ and $\g^{-\theta,-\omega}$, respectively. The conjugacy of maximal toral subalgebras in $\g^{-\omega}$ and $\g^{-\theta,-\omega}$ is a well-known fact about complex symmetric spaces $G/G^{\omega}$ and $H^*/H^{\omega}$ \cite[1.3]{symm}. The claims about conjugacy of maximal $\RR$-split toral subalgebras with $\ttt$ or $\ttt_1$ are reduced to conjugacy of their ``non-compact parts'' with $\ttt^+$ or~$\ttt_1^+$. The latter for maximal $\RR$-split toral subalgebras in $\g$ and $\g^{-\omega}$ is a classical fact about the Riemannian symmetric space $G^{\sigma}/G^{\sigma,\tau}$ \cite[Chap.\,4, Thm.\,4.1]{LieG}. And for maximal $\RR$-split toral subalgebras in $\g^{-\theta}$ and $\g^{-\theta,-\omega}$, it is \cite[Thm.\,1, Cor.\,3]{ss-symm}.
\end{proof}

\begin{corollary}\label{split}
If $G$ is $\RR$-split, then $\ttt$ is a Cartan subalgebra of $\g$ and $\ttt_1$ is a maximal toral subalgebra in $\g^{-\theta}$.
\end{corollary}

\begin{proof}
The only thing requiring an explanation is that $\ttt_1$ is a maximal toral subalgebra in $\g^{-\theta}$. Replacing $G$ with $Z_G(\ttt_1)$, we may assume that $\ttt_1$ is central in $\g$. Hence the derived subalgebra $\g'$ has zero intersection with $\g^{-\theta,-\omega}$. It follows that each indecomposable $(\theta,\omega)$-stable factor $\g_i\subseteq\g'$ sits either in $\g^{\theta}=\h$ or in $\g^{\omega}$. (Otherwise the kernel of the adjoint action of $\g_i^{\omega}$ on $\g_i^{-\omega}$ would be a non-trivial $(\theta,\omega)$-stable ideal in $\g_i$ containing $\g_i^{-\theta}$.) In the latter case $G_i$ is anisotropic (i.e., contains no $\RR$-split tori), whence $G$ is not $\RR$-split, a contradiction. Hence $\g\subseteq\h\oplus\ttt_1$, and the corollary follows.
\end{proof}

\begin{remark}
The arguments in the above proof show that, even if $G$ is not $\RR$-split,
\begin{align}
Z_G(T_1)          &= Z_H(T_1)^{\circ}          \cdot Z_{G^{\omega}}(T_1)^{\circ}     \cdot T_1,   \\
Z_{G^{\tau}}(T_1) &= Z_{H^{\tau}}(T_1)^{\circ} \cdot Z_{G^{\sigma,\tau}}(T_1)^{\circ} \cdot T_1^-.
\label{cent}
\end{align}
Indeed, by the above, this holds on the level of Lie algebras, and both $Z_G(T_1)$ and $Z_{G^{\tau}}(T_1)$ are connected.
\end{remark}

\subsection{Weyl groups}

The set of nonzero eigenweights of $\ttt_1$ in $\g$ (or of $\ttt_1^+$ in~$\g^{\sigma}$) is a (possibly non-reduced) root system. We call it the \emph{real root system} of $X$ and its Weyl group $W_0$ the \emph{real little Weyl group} of~$X$. If $G$ is $\RR$-split, then, by Corollary~\ref{split}, we get the usual restricted root system and the little Weyl group of the (complex) symmetric space $X$ \cite[\S4]{inv}.

\begin{lemma}\label{Weyl}
$W_0$ is identified with $N_S(\ttt_1)/Z_S(\ttt_1)=N_S(\ttt_1^+)/Z_S(\ttt_1^+)$, where $S$ is any of the groups $G,H,G^{\sigma},G^{\tau},H^{\tau},G^{\sigma,\tau}$.
\end{lemma}

\begin{proof}
For $S=G^{\sigma}$ or $G^{\sigma,\tau}$ this is \cite[Thm.\,5]{ss-symm}. In general, we have to prove two things.

First, we shall show that each root reflection $r_{\alpha}\in W_0$ comes from an element in $H^{\tau}$ or $G^{\sigma,\tau}$ normalizing $\ttt_1$. As in \cite{ss-symm}, we use $SL_2$-theory.

The root subspaces $\g_{\pm\alpha}$ are interchanged by $\theta,\omega$ and preserved by~$\theta^*$, hence decompose into the direct sum of $(\pm1)$-eigenspaces of $\theta^*$. The decompositions of any $\xi\in\g_{\alpha}^{\theta^*}$ and $\eta\in\g_{\alpha}^{-\theta^*}$ into the sums of eigenvectors for $\theta,\omega$ are given in terms of \eqref{8boxes} by the following picture:
$$
\begin{array}{|c|c|}
\hline
\eta_{01} &  \xi_{11} \\
\hline
 \xi_{00} & \eta_{10} \\
\hline
\end{array}
$$
In particular, $\theta$ and $\omega$ send $\xi$ to $\xi_{00}-\xi_{11}\in\g_{-\alpha}^{\theta^*}$ and $\eta$ to ${\pm\eta_{01}\mp\eta_{10}\in\g_{-\alpha}^{-\theta^*}}$, respectively. For any $\zeta\in\ttt_1$ we have
\begin{align*}
[\zeta,\xi_{00}] &= \alpha(\zeta)\xi_{11}, & [\zeta,\eta_{01}] &= \alpha(\zeta)\eta_{10}, \\
[\zeta,\xi_{11}] &= \alpha(\zeta)\xi_{00}, & [\zeta,\eta_{10}] &= \alpha(\zeta)\eta_{01}.
\end{align*}

Choose $\xi\ne0$ or $\eta\ne0$, depending on which of the subspaces $\g_{\alpha}^{\pm\theta^*}$ is nonzero. Since these subspaces are $\sigma$-stable, we may assume $\xi,\eta\in\g^{\sigma}$. Let $\zeta\in\ttt_1^+$ be the coroot vector of $\alpha$. Then we have an $\RR$-split $\sgl_2$-sub\-al\-ge\-bra
$$\s_{\alpha}=\langle\xi_{00},\xi_{11},\zeta\rangle\quad\text{or}\quad\langle\eta_{01},\eta_{10},\zeta\rangle$$
stable under all involutions. In terms of \eqref{8boxes}, its structure is represented in the following picture:
$$
\begin{array}{|c|c|}
\hline
\text{\raisebox{-1ex}{\huge{$0$}}} &
\begin{matrix}
 z & y \\
 y &-z \\
\end{matrix} \\
\hline
\begin{matrix}
 0 &-x \\
 x & 0 \\
\end{matrix} &
\text{\raisebox{-1ex}{\huge{$0$}}} \\
\hline
\end{array}
\qquad\text{or}\qquad
\begin{array}{|c|c|}
\hline
\begin{matrix}
 0 & y \\
 y & 0 \\
\end{matrix} &
\begin{matrix}
 z & 0 \\
 0 &-z \\
\end{matrix} \\
\hline
\text{\raisebox{-1ex}{\huge{$0$}}} &
\begin{matrix}
 0 &-x \\
 x & 0 \\
\end{matrix} \\
\hline
\end{array}
$$
Here each of the four boxes represents the part of $\sgl_2$ belonging to the respective common eigenspace of $\theta$ and $\omega$. The real structure $\sigma$ on $\sgl_2$ is the standard one given by complex conjugation of matrix entries.

Now $r_{\alpha}$ is represented both by
\begin{align*}
\begin{pmatrix}
 0 &-1 \\
 1 & 0 \\
\end{pmatrix} &\in H^{\sigma,\tau}\text{ (in the first case) or }G^{\sigma,\tau}\text{ (in the second case)}\\
\intertext{and by}
\begin{pmatrix}
 0  & \ii \\
\ii &  0  \\
\end{pmatrix} &\in H^{\tau}\text{ (in the second case).}
\end{align*}

Secondly, $N_G(\ttt_1)/Z_G(\ttt_1)$ is embedded in $W_0$. Indeed, $N_G(\ttt_1)$ acts on $\ttt_1$ as a subgroup of the ``big'' Weyl group of a Cartan subalgebra in $\g$ containing $\ttt_1$. (We can replace any element of $N_G(\ttt_1)$ with another one normalizing the Cartan subalgebra in the same coset modulo $Z_G(\ttt_1)$.) Since the Weyl chambers of the ``big'' root system of $G$ intersect $\ttt_1^+$ in the Weyl chambers of the real root system of $X$, the orbits of $N_G(\ttt_1)/Z_G(\ttt_1)$ intersect them in single points, as the $W_0$-orbits do. Hence $N_G(\ttt_1)/Z_G(\ttt_1)$ cannot be bigger than $W_0$.
\end{proof}

The nonzero eigenweights of $\ttt_1$ in $\g^{\theta^*}$ (or of $\ttt_1^+$ in $\g^{\sigma,\tau^*}$) form a root subsystem in the real root system of $X$. It is just the root system of the complex symmetric space $H^*/H^{\omega}$ or of the Riemannian symmetric space $G^{\sigma,\tau^*}/H^{\sigma,\tau}$ in the classical sense. We call its Weyl group $W_{00}$ the \emph{very little Weyl group} of $X$. It is a classical fact about symmetric spaces that
$$W_{00}\simeq N_S(\ttt_1)/Z_S(\ttt_1)=N_S(\ttt_1^+)/Z_S(\ttt_1^+),$$
where $S$ is any of the groups $H^*,H^{\omega},G^{\sigma,\tau^*},H^{\sigma,\tau}$.

\begin{example}\label{qforms}
Take $G=SL_n$ with the standard split real structure and $H=SO_{p,q}$ ($p+q=n$). The above involutions are defined by their action on $\g$:
\begin{gather*}
\sigma(\xi)=\bar\xi,\qquad\tau(\xi)=-{\bar\xi}^{\,\top},\\
\theta(\xi)=-I_{p,q}\xi^{\top}I_{p,q},\qquad\omega(\xi)=-\xi^{\top},\qquad\theta^*(\xi)=I_{p,q}\xi I_{p,q},
\end{gather*}
where
$$I_{p,q}=\diag(\underbrace{1,\dots,1}_p,\underbrace{-1,\dots,-1}_q).$$
The following picture describes the structure of the common eigenspaces for the involutions $\theta,\omega,\theta^*$ in the spirit of \eqref{8boxes}:
$$
\begin{array}{|c|c|}
\hline
\begin{array}{cc}
  0      & Y\strut \\
Y^{\top} & 0
\end{array} &
\begin{array}{cc}
\sym_p &   0    \\
  0    & \sym_q
\end{array} \\
\hline
\begin{array}{cc}
\so_p &   0   \\
  0   & \so_q
\end{array} &
\begin{array}{cc}
   0      & X \\
-X^{\top} & 0
\end{array} \\
\hline
\end{array}\ ,
$$
where $X,Y$ run over all matrices of size $p\times q$, and $\sym_m$ denotes the space of symmetric matrices of size $m\times m$. Here we have
$$H^*=S(L_p\times L_q),\qquad H^{\omega}=S(O_p\times O_q).$$
$T_1=T$ is the torus of diagonal matrices with determinant $1$, on which $W_0=S_n$ and $W_{00}=S_p\times S_q$ act by permutations of diagonal entries.
\end{example}

\subsection{Polar decompositions}

\begin{theorem}\label{polar}
There are decompositions:
\begin{align*}
G^{\sigma} &= G^{\sigma,\tau}\cdot(H^*)^{\sigma}\cdot H^{\sigma} = H^{\sigma,\tau}\cdot\exp\g^{\sigma,-\theta,\tau}\cdot\exp\g^{\sigma,-\theta,-\tau}\cdot\exp\h^{\sigma,-\tau} = \\
           &= G^{\sigma,\tau}\cdot T_1^+\cdot H^{\sigma}, \\
G^{\tau} &= G^{\sigma,\tau}\cdot(H^*)^{\tau}\cdot H^{\tau} = H^{\sigma,\tau}\cdot\exp\g^{\sigma,-\theta,\tau}\cdot\exp\ii\g^{\sigma,-\theta,-\tau}\cdot\exp\ii\h^{\sigma,-\tau} = \\     &= G^{\sigma,\tau}\cdot T_1^-\cdot H^{\tau}.
\end{align*}
\end{theorem}

\begin{remark}
The first collection of equalities (which we do not use in the sequel and provide for the sake of completeness) is a theorem of Mostow together with the Cartan decompositions of the non-compact real reductive groups $(H^*)^{\sigma},H^{\sigma}$ and of the compact symmetric pair $(G^{\sigma,\tau},H^{\sigma,\tau})$, see \cite[Thm.\,4.1(i)]{sph-fun}, \cite[Thms.\,9--10]{ss-symm}, \cite[7.1.2]{hyper}.

The second collection of equalities goes back to Hoogenboom \cite[Chap.\,6]{int-fun}. An elegant proof can be found in \cite[\S3]{2inv}. We are indebted to the referee for these references. This result is also claimed in \cite[Lemma II.6.2]{comp-symm}, but the proof is incomplete (namely the second equality in (II.6.5) is unjustified.) For the reader's convenience, we provide a self-contained elementary proof below.
\end{remark}

\begin{proof}[Proof of Theorem~\ref{polar}]
We only derive the decompositions of $G^{\tau}$. They follow by combining the Cartan decompositions of the compact symmetric pairs $(G^{\sigma,\tau},H^{\sigma,\tau})$, $((H^*)^{\tau},H^{\sigma,\tau})$, $(H^{\tau},H^{\sigma,\tau})$ with an observation that the multiplication map
$$\mu:G^{\sigma,\tau}\times T_1^-\times H^{\tau}\to G^{\tau}$$
is surjective.

To prove the surjectivity of $\mu$, we study the set of critical points of $\mu$ and its image, i.e., the set of critical values $\crit{G^{\tau}}\subset G^{\tau}$. We show that $\crit{G^{\tau}}$ is a union of finitely many locally closed submanifolds of codimension $\ge2$ in $G^{\tau}$. Once this is proved, we can conclude the proof as follows.

Outside of the critical set, $\mu$ is submersive, whence open. Also $\mu$ is closed being defined on a compact manifold. Since $G^{\tau}$ is connected, the set $G^{\tau}\setminus\crit{G^{\tau}}$ is open and connected. It intersects $\Img\mu$ in an open and (relatively) closed subset, whence $G^{\tau}\setminus\crit{G^{\tau}}\subset\Img\mu$. Now the density of $G^{\tau}\setminus\crit{G^{\tau}}$ and the closedness of $\Img\mu$ in $G^{\tau}$ will imply $\Img\mu=G^{\tau}$.

To pursue this strategy, we start with the critical points of $\mu$. By the $(G^{\sigma,\tau}\times H^{\tau})$-equiv\-ari\-ance, it suffices to consider points of the form $(1,t,1)$, $t=\exp\ii\gamma$, $\gamma\in\ttt_1^+$. The differential of $\mu$ at this point is given by
$$
d\mu:(\xi,\zeta\cdot t,\eta)\mapsto\xi\cdot t+\zeta\cdot t+t\cdot\eta
\qquad(\xi\in\g^{\sigma,\tau},\ \zeta\in\ttt_1^-,\ \eta\in\h^{\tau}).
$$
Hence $\Ker d\mu$ is defined by the equation $\xi+\zeta+\Ad(t)\eta=0$. Note that $$\Ad(t)\eta=\exp\ii\ad(\gamma)\,\eta=\cos\ad(\gamma)\,\eta+\ii\sin\ad(\gamma)\,\eta,$$
where the first summand is in $\h^{\tau}$, and the second summand is in $\g^{-\theta,\tau}$. The decompositions of $\xi,\eta,\zeta,\gamma$ into the sums of eigenvectors for $\theta,\omega$ are given in terms of \eqref{8boxes} by the following picture:
$$
\begin{array}{|c|c|}
\hline
          \eta_{01} &  \zeta,\gamma \\
\hline
 \xi_{00},\eta_{00} &      \xi_{10} \\
\hline
\end{array}
$$
Thus the defining equation of $\Ker d\mu$ reads as
\begin{align*}
\xi_{00}&=-\cos\ad(\gamma)\,\eta_{00},    &    \cos\ad(\gamma)\,\eta_{01}&=0,      \\
\xi_{10}&=-\ii\sin\ad(\gamma)\,\eta_{01}, & \ii\sin\ad(\gamma)\,\eta_{00}&=-\zeta.
\end{align*}
Since $\ad(\gamma)$ is diagonalizable, from $\zeta\in\Ker\ad(\gamma)\cap\Img\ad(\gamma)$ we get $\zeta=0$, $\xi=-\Ad(t)\eta$.

It follows that $(G^{\sigma,\tau}\times H^{\tau})$-orbits in $G^{\tau}$ intersect $T_1^-$ transversally and $\Ker d\mu$ is given by
\begin{align*}
\eta_{00}&\in\bigoplus_{\alpha(\gamma)=\pi{k},\ k\in\ZZ} (\g_{\alpha}^{\sigma}\oplus\g_{-\alpha}^{\sigma})^{\theta,\tau}\oplus\z_{\h^{\sigma,\tau}}(\ttt_1), \\
\eta_{01}&\in\bigoplus_{\substack{\alpha(\gamma)=\pi/2+\pi{k},\\ k\in\ZZ}} (\g_{\alpha}^{-\sigma}\oplus\g_{-\alpha}^{-\sigma})^{\theta,\tau}.
\end{align*}
(Recall that the root subspaces $\g_{\pm\alpha}$ are interchanged by $\theta$ and $\tau$, and preserved by $\sigma$ and $\theta^*=\theta\sigma\tau$.) Thus the critical set of $\mu$ is
$$G^{\sigma,\tau}\times\crit{T_1^-}\times H^{\tau},$$
where $\crit{T_1^-}\subset T_1^-$ is a union of finitely many hypersurfaces defined by equations $\alpha(t)=\pm1$ for $\alpha$ such that $\g_{\alpha}^{\theta^*}\ne0$ and $\alpha(t)=\pm\ii$ for $\alpha$ such that $\g_{\alpha}^{-\theta^*}\ne0$.

This hypersurface arrangement yields a natural stratification of $\Img\mu$ by locally closed submanifolds, which agrees with the orbit type stratification. Outside of the critical set, the $(G^{\sigma,\tau}\times H^{\tau})$-orbits have maximal dimension. The image $\crit{G^{\tau}}=G^{\sigma,\tau}\cdot\crit{T_1^-}\cdot H^{\tau}$ of the critical set is the union of strata of codimension $\ge2$, because the strata of $\crit{T_1^-}$ as well as the orbits intersecting them have smaller dimension. This observation completes the proof.
\end{proof}

\section{Review of Galois cohomology}
\label{Galois}

Our basic reference on Galois cohomology is Serre's book \cite{Galois}. We only consider Galois cohomology over real numbers, in which case the exposition is greatly simplified.

In this section, $G$ may be any linear algebraic group defined over $\RR$ (not necessarily connected or semisimple) and $X$ may be any homogeneous $G$-space (not necessarily symmetric) defined over $\RR$ and containing real points; $H$ still denotes the stabilizer of the base point $x_0\in X(\RR)$.

A Galois 1-cocycle with coefficients in $G$ is an element $z\in G$ which is inverse to its complex conjugate: $z\bar{z}=1$. The set $\Z1(\RR,G)$ of all Galois 1-cocycles is endowed with a $G$-action by twisted conjugation:
$$g:z\mapsto gz\bar{g}^{-1}.$$
The orbits $[z]$ of this action are the Galois cohomology classes and the orbit set
$$\Ho1(\RR,G)=\Z1(\RR,G)/G$$
is the (1-st) Galois cohomology set with coefficients in $G$. It is a pointed set, the base point being $[1]$, the class of the trivial 1-cocycle.

There is an exact sequence of pointed sets \cite[Chap.\,I, 5.4]{Galois}:
$$1\to H(\RR)\to G(\RR)\to X(\RR)\to\Ho1(\RR,H)\to\Ho1(\RR,G).$$
The maps are obvious, except for the next-to-last one, given by
\begin{equation}\label{connect}
x=gx_0\mapsto[\bar{g}^{-1}g].
\end{equation}
Exactness means, as usual, that the image of each map equals the kernel (i.e., the preimage of the base point) of the subsequent map. Note that the triviality of the kernel does not readily imply injectivity. The base points in $H(\RR)$, $G(\RR)$, and $X(\RR)$ are $1$, $1$, and $x_0$, respectively. These sets are usually referred to as the 0-th Galois cohomology sets with coefficients in $H$, $G$, and $X\simeq G/H$, respectively. Then the above exact sequence can be viewed as the long exact cohomology sequence corresponding to the short exact sequence of coefficients. It can be extended to the right by $\Ho1(\RR,G/H)$ if $H$ is a normal subgroup of $G$, and even further in some cases, see \cite[Chap.\,I, 5.5--5.7]{Galois}.

\begin{theorem}[{\cite[I.5.4, Cor.\,1]{Galois}}]\label{orbits}
The $G(\RR)$-orbits on $X(\RR)$ are the fibers of the map \eqref{connect}, i.e.,
$$X(\RR)/G(\RR)\simeq\Ker[\Ho1(\RR,H)\to\Ho1(\RR,G)].$$
\end{theorem}

\begin{proof}
Let $x_i=g_ix_0\in X(\RR)$ ($i=1,2$). Then $[\bar{g_1}^{-1}g_1]=[\bar{g_2}^{-1}g_2]$ in $\Ho1(\RR,H)$ if and only if $\bar{g_1}^{-1}g_1=\bar{h}^{-1}\bar{g_2}^{-1}g_2h$ for some $h\in H$, i.e., $g=g_2hg_1^{-1}\in G(\RR)$. But elements $g$ of this form are exactly those for which $gx_1=x_2$.
\end{proof}

Now let $G$ be semisimple (or, more generally, reductive) and $\tau$ be a Cartan involution on $G$ compatible with the real structure $\sigma$.

\begin{lemma}[cf.\ {\cite[Thm.\,5.3]{Galois&quot}, \cite[II.5.5]{comp-symm}}]\label{unitary}
Each Galois cohomology class $[z]\in\Ho1(\RR,G)$ is represented by a ``unitary'' 1-cocycle $z\in G^{\tau}$. The equivalence of cocycles of this form is given by twisted conjugation in $G^{\tau}$. In other words,
$$\Ho1(\RR,G)\simeq\Ho1(\langle\sigma\rangle,G^{\tau}),$$
the non-Abelian cohomology group of the cyclic group of order two $\langle\sigma\rangle$ with coefficients in $G^{\tau}$.
\end{lemma}

\begin{remark}
In \cite[II.5.5]{comp-symm}, it is only proved that the natural map $$\Ho1(\langle\sigma\rangle,G^{\tau})\to\Ho1(\RR,G)$$ is surjective with trivial kernel. We fill the gap in that proof below. Bremigan proves a more general result on relative Galois cohomology in \cite[Thm.\,5.3]{Galois&quot}, but the proof is more involved.
\end{remark}

\begin{proof}[Proof of Lemma~\ref{unitary}]
Take any 1-cocycle $z\in G$. We use the Cartan decomposition $G=G^{\tau}\cdot\exp\ii\g^{\tau}$ to write $z=z_-z_+$. Then
$$z\bar{z}=z_-z_+\overline{z}_-\overline{z}_+=1\implies z_-\overline{z}_-\cdot(\overline{z}_-)^{-1}z_+\overline{z}_-=(\overline{z}_+)^{-1}.$$
By the uniqueness of the Cartan decomposition,
$$z_-\overline{z}_-=1 \quad\text{and}\quad (\overline{z}_-)^{-1}z_+\overline{z}_-=z_-z_+z_-^{-1}=(\overline{z}_+)^{-1}.$$
Furthermore, logarithming and exponentiating back yields that
$$z_-z_+^tz_-^{-1}=(\overline{z}_+)^{-t}$$
for any element $z_+^t$ ($t\in\RR$) in the 1-parameter subgroup in $\exp\ii\g^{\tau}$ passing through $z_+$. Taking $g=(\overline{z}_+)^{1/2}$, we see that $gz\bar{g}^{-1}=z_-$.

Now suppose $z\in G^{\tau}$ and consider an equivalent cocycle $z'=gz\bar{g}^{-1}\in G^{\tau}$ (${g\in G}$). We have
\begin{multline*}
z'=g_-g_+z(\overline{g}_+)^{-1}(\overline{g}_-)^{-1}=\\=
g_-z(\overline{g}_-)^{-1}\cdot\overline{g}_-z^{-1}g_+z(\overline{g}_-)^{-1}\cdot
\overline{g}_-(\overline{g}_+)^{-1}(\overline{g}_-)^{-1} \implies \\
\implies z'\cdot\overline{g}_-\overline{g}_+(\overline{g}_-)^{-1}=
g_-z(\overline{g}_-)^{-1}\cdot\overline{g}_-z^{-1}g_+z(\overline{g}_-)^{-1}.
\end{multline*}
Again by the uniqueness of the Cartan decomposition, we get $z'=g_-z(\overline{g}_-)^{-1}$.
\end{proof}

\section{Results}
\label{results}

We now come to a description of the real group orbits in the real locus of a symmetric space. In the following, we prefer to denote the complex conjugation on $G$ by $\sigma$, rather than by a bar, in view of the presence of other involutions (both holomorphic and antiholomorphic).

\subsection{Galois cohomology description of real orbits}
\label{Galois-orb}

Consider the set $T_1\cap H$. Since $\theta$ acts on $T_1$ as an inversion and on $H$ identically modulo $Z(G)$, we have $s^2\in Z(G)$, $\forall s\in T_1\cap H$, whence $T_1\cap H$ is a finite subgroup of $T_1^-$. It contains (is contained in) $\0{T_1}2$ whenever $H\supseteq G^{\theta}$ (resp.\ $H\subseteq G^{\theta}$).

Since $\sigma$ acts on $T_1^-$ as an inversion, each $s\in T_1\cap H$ is represented as $s=t^2=t\sigma(t)^{-1}$ for some $t\in T_1^-$, i.e., it is a coboundary in $G$ and a cocycle in $H$. Hence $T_1\cap H$ maps to $\Ker[\Ho1(\RR,H)\to\Ho1(\RR,G)]$.

Consider the subgroups
\begin{align*}
N_0 &= \{n\in N_H(T_1)\mid n\sigma(n)^{-1}\in T_1\},\\
Z_0 &= \{z\in Z_H(T_1)\mid z\sigma(z)^{-1}\in T_1\}.
\end{align*}
It follows from the proof of Lemma~\ref{Weyl} that $N_0/Z_0=W_0$. Indeed, it suffices to represent the root reflections by elements of $N_0$. But this was actually done in the course of the proof.

Also note that $N_0=N_0^{\tau}\cdot Z_0^{\sigma,-\tau}$. Indeed, considering the Cartan decomposition $n=n_-n_+$ of an element $n\in N_0$ in $N_H(T_1)$ yields
\begin{multline*}
n\sigma(n)^{-1}=n_-n_+\sigma(n_+)^{-1}\sigma(n_-)^{-1}=s\in T_1\cap H\subset T_1^-
\implies \\ \implies n_-\sigma(n_-)^{-1}\cdot\sigma(n_-)n_+\sigma(n_-)^{-1}=s\cdot\sigma(n_-)\sigma(n_+)\sigma(n_-)^{-1}
\implies \\ \implies
n_-\sigma(n_-)^{-1}=s,\quad n_+=\sigma(n_+)\in Z_0.
\end{multline*}

The group $N_0$ acts on $T_1$ by twisted conjugation preserving $T_1\cap H$:
$$n:t\mapsto nt\sigma(n)^{-1}=ntn^{-1}\cdot n\sigma(n)^{-1}.$$
The ``Hermitian'' part $Z_0^{\sigma,-\tau}\subset N_0$ acts trivially.

\begin{remark}
The twisted conjugation action is not reduced to the $W_0$-ac\-tion, as the usual conjugation action of $N_0$. As an example, consider $G=SL_2$, $H=SO_{1,1}$ (cf.~Example~\ref{qforms}). Here
\begin{align*}
T_1&=\left\{
\begin{pmatrix}
u & 0      \\
0 & u^{-1} \\
\end{pmatrix}
\right|\left.
\vphantom{\begin{pmatrix}
u & 0      \\
0 & u^{-1} \\
\end{pmatrix}}
u\in\CC^{\times}
\right\}, \\
N_0&=\left\{
\begin{pmatrix}
\pm1 &   0  \\
  0  & \pm1 \\
\end{pmatrix},
\begin{pmatrix}
   0   & \pm\ii \\
\pm\ii &    0   \\
\end{pmatrix}
\right\}.
\end{align*}
Conjugation by
$$n=\begin{pmatrix}
 0  & \ii \\
\ii &  0  \\
\end{pmatrix}$$
sends $u$ to $u^{-1}$, while twisted conjugation sends $u$ to $-u^{-1}$. Thus $W_0$ has two orbits on $T_1\cap H= \{\pm{E}\}$, while $N_0$ has just one orbit.
\end{remark}

\begin{theorem}\label{Ker}
$\Ker[\Ho1(\RR,H)\to\Ho1(\RR,G)]\simeq(T_1\cap H)/N_0$ (the quotient by the twisted conjugation action) is in bijection with the orbit set for $G(\RR)$ acting on $X(\RR)$.
\end{theorem}

\begin{remark}
In \cite[II.6.5]{comp-symm}, it is claimed that the set of $G(\RR)$-orbits in $X(\RR)$ is identified with the quotient of $\0{T_1}2$ (which coincides with $T_1\cap H$ for $H=G^{\theta}$) modulo a group smaller than~$N_0$, that is $N_0^{\sigma,\tau}=N_{H^{\sigma,\tau}}(T_1)$. For this latter group, the twisted conjugation action coincides with the usual one and is reduced to the action of the very little Weyl group~$W_{00}$. But this is false in general.

Indeed, in Example~\ref{qforms}, $X$ is the space of quadratic forms in $n$ variables with discriminant~$(-1)^q$. The base point $x_0\in X$ is the quadratic form with the matrix $I_{p,q}$. The orbits of $G(\RR)$ in $X(\RR)$ consist of quadratic forms of given signature $(p',q')$ with $q'\equiv q\pmod2$, the number of  orbits is $\sim n/2$.

However, $T_1=T$ is the torus of diagonal matrices with $\det=1$ and $\0{T_1}2=T_1\cap H$ is the set of matrices of the form $\diag(\pm1,\dots,\pm1)$ with even number of minuses. The group $W_{00}\simeq S_p\times S_q$ acts on $\0{T_1}2$ by permutations of the first $p$ and last $q$ diagonal entries, the number of orbits being $\sim pq/2$.
\end{remark}

\begin{proof}[Proof of Theorem~\ref{Ker}]
Take any $[z]\in\Ker[\Ho1(\RR,H)\to\Ho1(\RR,G)]$. By Lemma~\ref{unitary}, we may assume that $z=g\sigma(g)^{-1}\in H^{\tau}$ with $g\in G^{\tau}$. By Theorem~\ref{polar}, $g=htk$ with $h\in H^{\tau}$, $t\in T_1^-$, $k\in G^{\sigma,\tau}$, whence
$$z=htk\cdot k^{-1}t\sigma(h)^{-1}=ht^2\sigma(h)^{-1}\sim t^2=s\text{ in }\Z1(\RR,H)$$
and $s\in T_1^-\cap H$. Thus any cohomology class in the kernel is represented by a 1-cocycle in~$T_1\cap H$.

Now suppose we are given two equivalent cocycles $s,s'\in T_1\cap H$, $s=hs'\sigma(h)^{-1}$, $h\in H$. Define a new real structure on $G$ by
$$\tilde\sigma(g)=s\sigma(g)s^{-1}$$
(which is indeed an involution, because $\sigma(s)=s^{-1}$). It still commutes with all other involutions.

Observe that both $T_1$ and $hT_1h^{-1}$ are maximal $(\tilde\sigma,\theta)$-split tori in $G$. Indeed, for any $t\in T_1$ we have
\begin{align*}
\tilde\sigma(t)        &= \sigma(t), \\
      \theta(t)        &= t^{-1}, \\
\tilde\sigma(hth^{-1}) &= s\sigma(h)\sigma(t)\sigma(h)^{-1}s^{-1}=hs'\sigma(t)(s')^{-1}h^{-1}=h\sigma(t)h^{-1}, \\
      \theta(hth^{-1}) &= \theta(h)\theta(t)\theta(h)^{-1} = ht^{-1}h^{-1},
\end{align*}
i.e., both tori are $(\tilde\sigma,\theta)$-split. By Theorem~\ref{conj}, $T_1$ is a maximal $(\sigma,\theta)$-split torus in $G$. But $\tilde\sigma=\sigma$ on $Z_G(T_1)$, because $s\in T_1$. Hence $T_1$ cannot be extended to a bigger $(\tilde\sigma,\theta)$-split torus. The same is true for $hT_1h^{-1}$ by dimension reasons.

By Theorem~\ref{conj}, there exists $g\in H^{\tilde\sigma}$ such that $ghT_1h^{-1}g^{-1}=T_1$. Then $n=gh\in N_H(T_1)$ and
$$ns'\sigma(n)^{-1}=ghs'\sigma(h)^{-1}\sigma(g)^{-1}=gs\sigma(g)^{-1}=g\tilde\sigma(g)^{-1}s=s.$$
Furthermore, $ns'\sigma(n)^{-1}=ns'n^{-1}\cdot n\sigma(n)^{-1}$ and, since the last product and its first factor both belong to $T_1\cap H$, we have $n\sigma(n)^{-1}\in T_1\cap H$, whence $n\in N_0$.
\end{proof}

\subsection{Geometric description of real orbits}
\label{geom-orb}

Put $Z=T_1x_0=Tx_0\simeq A=T/(T\cap H)=T_1/(T_1\cap H)$.

As noted above, each 1-cocycle $s\in T_1\cap H$ can be represented as $s=t^2=\sigma(t)^{-1}t$ for some $t\in T_1^-$. By \eqref{connect}, it corresponds to the $G(\RR)$-orbit of $x=tx_0$. Hence all $G(\RR)$-orbits intersect~$Z(\RR)$.

The set $Z(\RR)$ is identified with $A(\RR)=A^+\times\0A2$ via the orbit map at $x_0$, and we may consider a subset $\0Z2\subset Z(\RR)$ corresponding to $\0A2$. It is a set of orbit representatives (cross-section) for the action of $T_1^+$ or $T^+=T_0^+\times T_1^+$ on $Z(\RR)$.

The remaining part $\0T2$ of $T(\RR)$ preserves $\0Z2$. Clearly, the $T(\RR)$-or\-bits on $Z(\RR)$ intersect $\0Z2$ in $\0T2$-orbits.

There is a natural bijection between $\0Z2/\0{T_1}2$ and $(T_1\cap H)/(T_1\cap H)^2$ given by sending $x=tx_0$ ($t\in T_1^-$, $t^2\in T_1\cap H$) to $s=t^2$. Note that $\0{T_1}2$ acts on $Z$ trivially if $H\supseteq G^{\theta}$, and $(T_1\cap H)^2=\{1\}$ if $H\subseteq G^{\theta}$.

\begin{lemma}\label{Z0-action}
If $G$ is $\RR$-split, then there are natural bijections between the orbit sets for the following actions: $T(\RR)\curvearrowright Z(\RR)$, $\0T2\curvearrowright\0Z2$, $Z_0\curvearrowright T_1\cap H$ (by twisted conjugation).
\end{lemma}

\begin{proof}
It only remains to consider the third orbit set.

Suppose that $x_1=t_1x_0$, $x_2=t_2x_0\in\0Z2$ ($t_1,t_2\in T_1^-$) correspond to 1-cocycles $s_1=t_1^2$, $s_2=t_2^2\in T_1\cap H$, and $x_2=tx_1$ for some $t\in\0T2$. Then $t_2=tt_1t_0$ for some $t_0\in T^-\cap H$, whence $s_2=s_1t_0^2=t_0s_1\sigma(t_0)^{-1}$ is obtained from $s_1$ by twisted conjugation with $t_0\in Z_0$.

Conversely, suppose that $s_2=zs_1\sigma(z)^{-1}=s_1z\sigma(z)^{-1}$ for some $z\in Z_0$; we may even assume that $z\in Z_0^{\tau}$. Decomposing $z=z_+z_-$ with respect to the Cartan decomposition $G^{\tau}=\exp\ii\g^{\sigma,-\tau}\cdot G^{\sigma,\tau}$, we get
$$z\sigma(z)^{-1}=z_+z_-\sigma(z_-)^{-1}\sigma(z_+)^{-1}=z_+^2.$$
Moreover, by \eqref{cent} and since $Z_{G^{\sigma,\tau}}(T_1)^{\circ}\subseteq H^{\sigma,\tau}$ (because $G$ is $\RR$-split), we may assume that $z_-\in Z_H(T_1)^{\sigma,\tau}$, $z_+\in\exp\ii\z_{\h}(\ttt_1)^{\sigma,-\tau}\cdot(T_1\cap H)$. Then
$$z_+=kt_0k^{-1}\quad\text{for some $k\in Z_{H}(T_1)$, $t_0\in T_0^-\cdot(T_1\cap H)=T^-\cap H$}$$
(since $\ttt_0^+$ is a maximal Abelian subalgebra in $\z_{\h}(\ttt_1)^{\sigma,-\tau}$ and all such subalgebras are conjugate). But, on the other side,
$$z_+^2=s_2s_1^{-1}\in T_1\implies z_+^2=k^{-1}z_+^2k=t_0^2\in(T_0\cap T_1)\cdot(T_1\cap H)^2.$$
Then
\begin{equation*}
s_2=s_1t_0^2\implies t_2=tt_1t_0\implies x_2=tx_1\quad\text{for some }t\in\0T2.
\qedhere
\end{equation*}
\end{proof}

\begin{remark}
In the course of the proof, we have actually shown that the twisted conjugation action of $Z_0$ on $T_1\cap H$ is just translation by $(T_0\cap T_1)\cdot(T_1\cap H)^2\subseteq T_1\cap H$.
\end{remark}

Now consider the action of $N_0$ on $T_1\cap H$ by twisted conjugation:
$$n:s\mapsto ns\sigma(n)^{-1}=nsn^{-1}\cdot n\sigma(n)^{-1}=\1sn\cdot c(n),$$
where $\1{}n$ denotes the (usual) conjugation with $n$ and the formula $c(n)=n\sigma(n)^{-1}$ defines a map $c:N_0\to T_1\cap H$ satisfying the cocycle property:
$$c(mn)=c(m)\cdot\1{c(n)}m.$$

Let $a(n)\in T_1^-$ be such that $a(n)^2=c(n)$. Note that $a(n)$ is well defined only up to multiplication by $\0{T_1}2$. The map $a:N_0\to T_1/\0{T_1}2$ is a cocycle, too. The induced action of $N_0$ on $(T_1\cap H)/(T_1\cap H)^2$ corresponds to an action of $N_0$ on $\0Z2/\0{T_1}2$ given by
\begin{equation}\label{N0-action}
n:x=tx_0\mapsto a(n)\,\1tn\,x_0=a(n)nx=a(n)^{-1}nx.
\end{equation}
Note that $g=a(n)^{-1}n\in G(\RR)$ since
$$
\sigma(a(n)^{-1}n)=a(n)\sigma(n)=a(n)^{-1}c(n)\sigma(n)=a(n)^{-1}n.
$$

In view of Theorems \ref{orbits} and \ref{Ker}, we have proved:
\begin{theorem}\label{slice-orb}
The $G(\RR)$-orbits on $X(\RR)$ are in a bijective correspondence with the $N_0$-orbits on $\0Z2/\0{T_1}2$ for the action~\eqref{N0-action}. In fact, each $G(\RR)$-orbit intersects $\0Z2$ in the inverse image of an $N_0$-orbit under the quotient map $\0Z2\to\0Z2/\0{T_1}2$.
\end{theorem}

\begin{remark}
If $H\supseteq G^{\theta}$ (e.g., if $G$ is simply connected), then $a$ yields a well-defined cocycle map $N_0\to\0A2$. Formula~\eqref{N0-action} defines an action of $N_0$ on $\0Z2\simeq\0A2$; it is the usual conjugation action on $\0A2$, which factors through the $W_0$-action, composed with the shift by the cocycle~$a(\cdot)^{-1}$. The $G(\RR)$-orbits on $X(\RR)$ intersect $\0Z2$ in the $N_0$-orbits.
\end{remark}

\begin{example}
Let us describe the $G(\RR)$-orbits on $X(\RR)$ in Example~\ref{qforms}.

The group $Z_0=\0T2$ consists of matrices of the form $\diag(\pm1,\dots,\pm1)$ with even number of minuses. The group $N_0$ is generated by $Z_0$ and the elements representing root reflections, which act on the standard orthonormal basis $e_1,\dots,e_n$ of $\CC^n$ by fixing all $e_k$ with $k\ne i,j$ ($i<j$) and on $e_i,e_j$ by the matrix
\begin{align*}
\begin{pmatrix}
0 & -1 \\
1 &  0 \\
\end{pmatrix}
&\quad\text{if $i,j\le p$ or $i,j>p$,}\\
\intertext{and by}
\begin{pmatrix}
 0  & \ii \\
\ii &  0  \\
\end{pmatrix}
&\quad\text{if $i\le p<j$.}
\end{align*}
The action of $Z_0$ on $T_1=T$ is trivial, and the root reflections act by transposing the $i$-th and $j$-th diagonal entries, and changing their signs if $i\le p<j$. Therefore $N_0$ acts on $T_1\cap H=\0T2$ by permutations of the first $p$ and last $q$ diagonal entries and by replacing a pair of equal entries at the $i$-th and $j$-th places with the opposite values for any pair of indices $i\le p<j$. It is clear that the $N_0$-action takes each point of $T_1\cap H$ to a unique canonical form
\begin{align*}
s_k&=\diag(\underbrace{1,\dots,1}_{p-2k},\underbrace{-1,\dots,-1}_{2k},\underbrace{1,\dots,1}_q)\\
\intertext{or}
s'_k&=\diag(\underbrace{1,\dots,1}_p,\underbrace{-1,\dots,-1}_{2k},\underbrace{1,\dots,1}_{q-2k}).
\end{align*}

On the geometric side, the base point of $X$ is the quadratic form
$$x_0=y_1^2+\dots+y_p^2-y_{p+1}^2-\dots-y_{p+q}^2,$$
where $y_1,\dots,y_n$ are the coordinates in the basis $e_1,\dots,e_n$ ($n=p+q$). The 1-cocycles $s_k=t_k^2$, $s'_k=(t_k')^2$ correspond to the quadratic forms
\begin{align*}
 t_kx_0&=y_1^2+\dots+y_{p-2k}-y_{p-2k+1}^2-\dots-y_{p+q}^2,\\
t'_kx_0&=y_1^2+\dots+y_{p+2k}-y_{p+2k+1}^2-\dots-y_{p+q}^2
\end{align*}
of signatures $(p-2k,q+2k)$ and $(p+2k,q-2k)$, respectively, where
\begin{align*}
t_k&=\diag(\underbrace{\pm1,\dots,\pm1}_{p-2k},\underbrace{\pm\ii,\dots,\pm\ii}_{2k},\underbrace{\pm1,\dots,\pm1}_q),\\
t'_k&=\diag(\underbrace{\pm1,\dots,\pm1}_p,\underbrace{\pm\ii,\dots,\pm\ii}_{2k},\underbrace{\pm1,\dots,\pm1}_{q-2k}).
\end{align*}
Clearly, these quadratic forms represent the $G(\RR)$-orbits on $X(\RR)$.
\end{example}

If $G$ is $\RR$-split, then $Z\subset X$ is the Brion--Luna--Vust slice \cite{BLV}. By Lemma~\ref{Z0-action} and Theorem~\ref{slice-orb}, the $G(\RR)$-orbits on $X(\RR)$ are in a bijective correspondence with the orbits for the $W_0$-action on $\0Z2/\0T2$ which is defined by composing the usual $W_0$-action on $\0A2$ with the shift by the cocycle $W_0\to T/\0T2(T\cap H)$ given by $a(\cdot)^{-1}$.

This observation puts our result into a broader context of spherical homogeneous spaces, which are defined by property of having an open orbit of a Borel subgroup of $G$, see, e.g., \cite[Chap.\,5]{hom&emb}. The concepts of Brion--Luna--Vust slice and little Weyl group make sense in this setting. We shall develop the classification of real orbits on real loci of spherical homogeneous spaces elsewhere.

\end{document}